\documentclass{amsart}

\usepackage[dvips]{graphicx}
\usepackage{xcolor}

\theoremstyle{plain}
\newtheorem{theorem}{Theorem}[section]
\newtheorem{lemma}[theorem]{Lemma}
\newtheorem{corollary}[theorem]{Corollary}

\theoremstyle{definition}
\newtheorem{definition}[theorem]{Definition}

\theoremstyle{remark}

\newtheorem*{acknowledgements}{Acknowledgements}

\newcommand{\del}{\partial}
\newcommand{\R}{\mathbb{R}}
\newcommand{\N}{\mathbb{N}}
\renewcommand{\H}{\mathbb{H}}
\renewcommand{\S}{\mathbb{S}}
\newcommand{\E}{\mathbb{E}}

\begin{document}

\title{Geometric bistellar moves relate geometric triangulations}

\author[Kalelkar]{Tejas Kalelkar}
\address{Mathematics Department, Indian Institute of Science Education and Research, Pune 411008, India}
\email{tejas@iiserpune.ac.in}

\author[Phanse]{Advait Phanse}
\address{Mathematics Department, Indian Institute of Science Education and Research, Pune 411008, India}
\email{advait.phanse@students.iiserpune.ac.in}

\date{\today}

\keywords{Hauptvermutung, geometric triangulation, bistellar moves, flip graph, combinatorial topology}

\subjclass[2010]{Primary 57Q25, 57Q20}

\begin{abstract}
A geometric triangulation of a Riemannian manifold is a triangulation where the interior of each simplex is totally geodesic. Bistellar moves are local changes to the triangulation which are higher dimensional versions of the flip operation of triangulations in a plane. We show that geometric triangulations of a compact hyperbolic, spherical or Euclidean manifold are connected by geometric bistellar moves (possibly adding or removing vertices), after taking sufficiently many derived subdivisions. For dimensions 2 and 3, we show that geometric triangulations of such manifolds are directly related by geometric bistellar moves (without having to take derived subdivision).
\end{abstract}

\maketitle

\section{Introduction and Notation}\label{intro}
If we do not allow adding or removing vertices, it is known that the flip graph of Euclidean triangulations of 2-dimensional polytopes is connected. This forms the basis for the Lawson edge flip algorithm to obtain a Delaunay triangulation. For 5-dimensional polytopes, Santos \cite{San} has given examples of triangulated polytopes with disconnected flip graphs. The problem of connectedness of such a flip graph for 3-dimensional polytopes is still open.

In this article we show that if we allow bistellar moves which add or remove vertices then the flip graph is connected in dimension 2 and 3 not just for polytopes but also for geometric triangulations of any compact Euclidean, spherical or hyperbolic manifold. This can be used to show that a quantity calculated from a geometric triangulation which is invariant under geometric bistellar moves is in fact an invariant of the manifold. Furthermore in dimension greater than 3, we show that any two such triangulations are related by bistellar moves after taking suitably many derived subdivisions.

In \cite{KalPha} we give an alternate approach using shellings to relate geometric triangulations via topological bistellar moves (so intermediate triangulations may not be  geometric). This gives a tighter bound on the number of such flips required to relate  the geometric triangulations, than the one which can be  calculated from the algorithm in this article. 

Note that a foundational theorem of Pachner\cite{Pac} states that PL triangulations of PL homeomorphic manifolds are related by topological bistellar moves. Whitehead\cite{Whi} has shown that smooth triangulations of diffeomorphic smooth manifold are related by smooth bistellar moves. In this context, our result shows that geometric triangulations of isometric constant curvature manifolds are related by geometric bistellar moves up to derived subdivisions.

A \emph{geometric triangulation} of a Riemannian manifold is a finite triangulation where the interior of each simplex is a totally geodesic disk. Every constant curvature manifold has a geometric triangulation. In the converse direction, Cartan has shown that if for every point $p$ in a Riemannian manifold $M$ and every subspace $V$ of $T_p M$ there exists a totally geodesic submanifold $S$ through $p$ with $T_pS = V$, then $M$ must have constant curvature; which seems to suggest that the only manifolds which have many geometric triangulations are the constant curvature ones. A \emph{common subcomplex} of simplicial triangulations $K_1$ and $K_2$ of $M$ is a simplicial complex structure $L$ on a subspace of $M$ such that $K_1|_{|L|}= K_2|_{|L|}=L$. The main result in this article is the following:
\begin{theorem}\label{mainthm1}
Let $K_1$ and $K_2$ be geometric simplicial triangulations of a compact constant curvature manifold $M$ with a (possibly empty) common subcomplex $L$ with $|L| \supset \del M$. When $M$ is spherical we assume that the diameter of the star of each simplex is less than $\pi$. Then for some $s\in\N$, the $s$-th derived subdivisions $\beta^s K_1$ and $\beta^s K_2$ are related by geometric bistellar moves which keep $\beta^s L$ fixed.
\end{theorem}

We call a $\Delta$-complex $K$ the geometric triangulation of a cusped hyperbolic manifold $M$ if for some subset $V'$ of the set of vertices of $K$, $M=|K| \setminus |V'|$ and the interior of each simplex of $K$ is a totally geodesic disk in $M$. Cusped finite volume hyperbolic manifolds have canonical ideal polyhedral decompositions \cite{EpsPen}. Further dividing this polyhedral decomposition into ideal tetrahedra without introducing new vertices may result in degenerate flat tetrahedra. If however we allow genuine vertices, simply taking a derived subdivision of this polyhedral decomposition gives a geometric triangulation for any cusped manifold. For cusped manifolds we have the following weaker result:

\begin{theorem}\label{mainthm2}
Let $K_1$ and $K_2$ be geometric simplicial triangulations of a cusped hyperbolic manifold which have a common geometric subdivision. Then for some $s\in \N$, the $s$-th derived subdivisions $\beta^s K_1$ and $\beta^s K_2$ are related by geometric bistellar moves.
\end{theorem}

In recent work \cite{KalRag}, we have shown that geometric triangulations of cusped hyperbolic 3-manifolds do in fact have a common geometric subdivision with a bounded number of simplexes and that such triangulations are related by a bounded number of geometric bistellar moves.

In dimension 2 and 3, every internal stellar move can be realised by geometric bistellar moves (see for example Lemma 2.11 of \cite{IzmSch}), so we get the following immediate corollary:
\begin{corollary}
Let $K_1$ and $K_2$ be geometric simplicial triangulations of a closed constant curvature 3-manifold $M$. When $M$ is spherical we assume that the diameter of the star of each simplex is less than $\pi$. Then $K_1$ is related to $K_2$ by geometric bistellar moves.
\end{corollary}

An \emph{abstract simplicial complex} consists of a finite set $K^0$ (the vertices) and a family $K$ of subsets of $K^0$ (the simplexes) such that if $B \subset A \in K$ then $B \in K$. A \emph{simplicial isomorphism} between simplicial complexes is a bijection between their vertices which induces a bijection between their simplexes. A \emph{realisation} of a simplicial complex $K$ is a subspace $|K|$ of some $\R^N$, where $K^0$  is represented by a finite subset of $\R^N$ and vertices of each simplex are in general position and represented by the linear simplex which is their convex hull. Every simplicial complex has a realisation in $\R^{N}$ where $N$ is the size of $K_0$, by representing $K_0$ as a basis of $\R^N$. Any two realisations of a simplicial complex are simplicially isomorphic. For $A$ a simplex of $K$, we denote by $\del A$ the \emph{boundary complex} of $A$. When the context is clear, we shall use the same symbol $A$ to denote the simplex and the simplicial complex $A \cup \del A$. We call $K$ a \emph{simplicial triangulation} of a manifold $M$ if there exists a homeomorphism from a realisation $|K|$ of $K$ to $M$. The simplexes of this triangulation are the images of simplexes of $|K|$ under this homeomorphism. The books by Rourke and Sanderson\cite{RouSan} and Ziegler\cite{Zie2} are good sources of introduction to the theory of piecewise linear topology. 

\begin{definition}
For $A$ and $B$ simplexes of a simplicial complex $K$, we denote their join $A\star B$ as the simplex $A \cup B$. As the join of totally geodesic disks in a constant curvature manifold gives a totally geodesic disk, operations involving joins are well-defined in the class of geometric triangulations of a constant curvature manifold.

The link of a simplex $A$ in a simplicial complex $K$ is the simplicial complex defined by $lk (A, K) = \{ B \in K : A \star B \in K \}$. The (closed) star of $A$ in $K$ is the simplicial complex defined by $st(A, K)=A \star lk (A, K)$. 
\end{definition}

\begin{definition}
Suppose that $A$ is an $r$-simplex in a simplicial complex $K$ of dimension $n$ then a stellar subdivision on $A$ gives the geometric triangulation $(A, a) K$ by replacing $st(A, K)$ with $a \star \del A \star lk(A, K)$ for $a\in int(A)$. The inverse of this operation $(A, a)^{-1} K$ is called a stellar weld and they both are together called stellar moves.

When $lk(A, K)=\del B$ for some $(n-r)$-dimensional geometric simplex $B \notin K$, then the Pachner or bistellar move $\kappa(A, B)$ consists of changing $K$ by replacing $A \star \del B$ with $\del A \star B$. Note that the bistellar move $\kappa(A, B)$ is the composition of a stellar subdivision and a stellar weld, namely $(B, a)^{-1} (A, a)$. Another way of defining this is to take a disk subcomplex $D$ of $K$ which is simplicially isomorphic to a disk $D'$ in $\del \Delta^{n+1}$ and the flip consists of replacing it with the disk  $\del \Delta^{n+1} \setminus int(D')$.

The derived subdivision $\beta K$ of $K$ is obtained from $K$ by performing a stellar subdivision on all $r$-simplexes, and ranging $r$ inductively from $n$ down to 1.
\end{definition}

All stellar and bistellar moves  we consider are geometric in nature. Not every combinatorial bistellar move in a geometric manifold can be expressed by geometric bistellar moves (see Fig \ref{combimove}). For details of geometric bistellar moves see Santos \cite{San}. Among cusped hyperbolic manifolds, it is remarked in \cite{DadDua} that the Figure Eight knot complement has geometric ideal triangulations which can not be related by geometric bistellar moves which do not introduce genuine vertices.

\begin{figure}
\begin{center}
\includegraphics[width=0.5\textwidth]{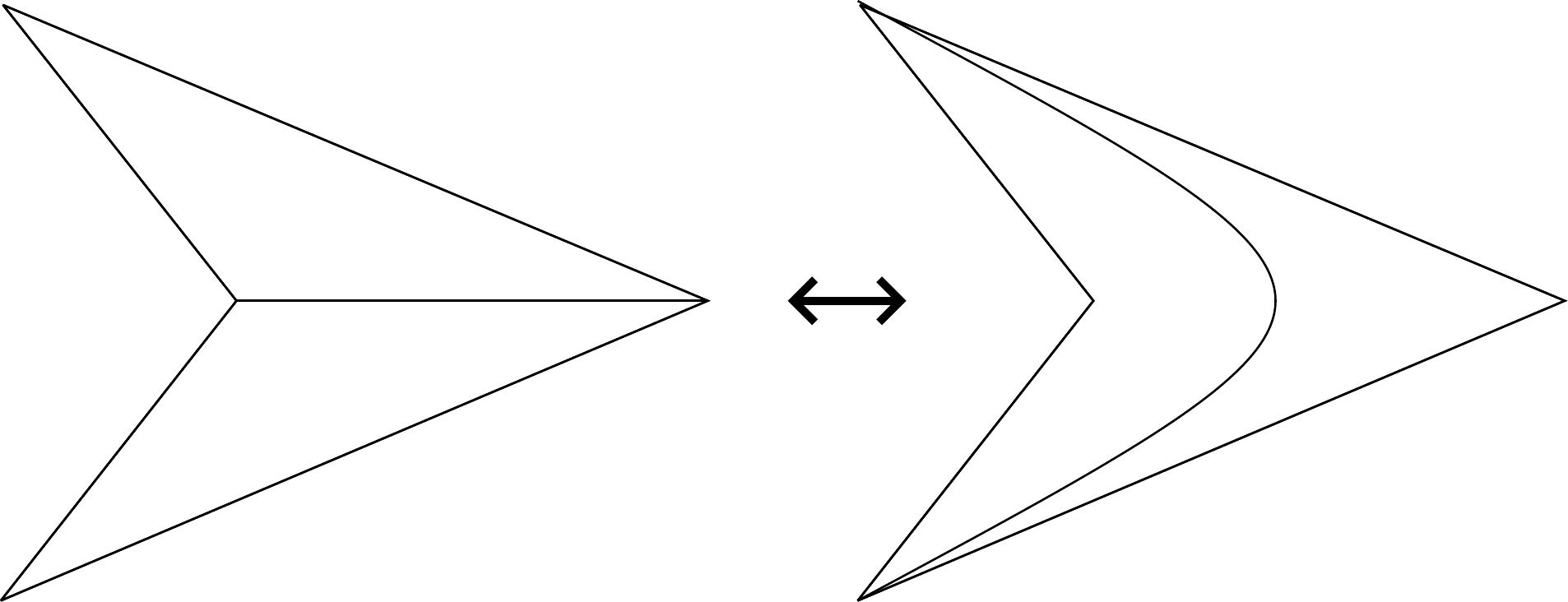}
\caption{A 2-2 combinatorial bistellar move which is not geometric.}\label{combimove}
\end{center}
\end{figure}

It is known that linear triangulations of a convex polytope $P \subset \R^N$ are related by stellar moves\cite{Mor}\cite{Wlo}. Our proof involves a straightforward extension of this result to star-convex polytopes and bistellar moves.

As the supports in $\R^N$ of two triangulations of a manifold may be different so when the manifold is not a polytope we can not take a linear cobordism between them. A subtle point here is that even if we obtain a common geometric refinement of two geometric triangulations, the refinement may not be a simplicial subdivision of the corresponding simplicial complexes. To see a topological refinement which is not a simplicial subdivision, observe that there exists a simplicial triangulation $K$ of a 3-simplex $\Delta$ which contains in its 1-skeleton a trefoil with just 3 edges \cite{Lic2}. If $K$ were a simplicial subdivision of $\Delta$ there would exist a linear embedding of $\Delta$ in some $\R^N$ which takes simplexes of $K$ to linear simplexes in $\R^N$. As the stick number of a trefoil is 6, there can exist no such embedding. While there may not exist such a global embedding of a geometric triangulation $K$ as a simplicial complex in $\R^N$ which takes geometric subdivisions to linear subdivisions, for constant curvature manifolds there does exist such a local embedding on stars of simplexes of $K$. This is the property we exploit in this note.

\section{Star-convex flat polyhedra}\label{starconvexsec}
\begin{definition}
We call a polyhedron $P$ in $\R^n$ strictly star-convex with respect to a point $x$ in its interior if for any $y \in P$, the interior of the segment $[x,y]$ lies in the interior of $P$. We call the polyhedron $P \subset \R^n$ flat if it is $n$-dimensional.

We call a triangulation $K$ of $P$ regular if there is a  function $h: |K| \to \R$ that is linear on each simplex of $K$ and strictly convex across codimension one simplexes of $K$, i.e., if points $x$ and $y$ are in neighboring top-dimensional simplexes of $K$ then the segment connecting $h(x)$ and $h(y)$ is above the graph of $h$ (except at the end points).
\end{definition}

In their proof of the weak Oda conjecture, Morelli \cite{Mor} and Wlodarczyk \cite{Wlo} proved that any two triangulations of a convex polyhedron are related by a sequence of stellar moves. As interior stellar moves can be given by bistellar moves in dimension  3, Izmestiev and Schlenker \cite{IzmSch} have improved on this result to show the following:
\begin{theorem} [Lemma 2.11 of \cite{IzmSch}]
Any two triangulations of a convex polyhedron $P$ in $\R^3$ can be connected by a sequence of geometric bistellar moves, boundary geometric stellar moves and continuous displacements of the interior vertices.
\end{theorem}

With their techniques however, even when the two triangulations agree on the boundary, we still need boundary stellar moves to relate them. Our aim in this section is to show that their techniques can be tweaked to give a boundary relative version for triangulations of strictly star-convex flat polyhedra in any dimension. The main theorem of this section is the following:

Our aim in this section is to show that their techniques also give a boundary relative version for triangulations of strictly star-convex flat polyhedra, with the stronger notion of bistellar equivalence in place of stellar equivalence. The main theorem of this section is the following:
\begin{theorem}\label{starconvexthm}
Let $P \subset \R^n$ be a strictly star-convex flat polyhedron. Let $K_1$ and $K_2$ be triangulations of $P$ that agree on the boundary. Let $\beta^s K_1$ and $\beta^s K_2$ denote $s$-th derived subdivisions of $K_1$ and $K_2$ which also agree on the boundary. Then there exists $s\in \N$, such that $\beta^s K_1$ and $\beta^s K_2$ are bistellar equivalent.
\end{theorem}

We use the following simple observation in the proof:
\begin{lemma}[Lemma 4, Ch 1 of \cite{Zee}]\label{derivedtorefined}
Let $K$ and $L$ denote two simplicial complexes with $|K| \subset |L|$. Then there exists $r\in\N$ and a subdivision $K'$ of $K$ such that $K'$ is a subcomplex of $\beta^r L$.
\end{lemma}

\begin{lemma}\label{regular}
Let $K$ denote a triangulated flat polyhedron. Then for some $s \in \N$, its $s$-th derived subdivision $\beta^s K$ is regular.
\end{lemma}
\begin{proof}
Let $\Delta$ be an $n$-simplex with $|\Delta| \supset |K|$. By Lemma \ref{derivedtorefined}, there exists an $r \in \N$ and subdivision $K'$ of $K$ which is a subcomplex of $\beta^r \Delta$. As $\Delta$ is trivially a regular triangulation, so its stellar subdivision $\beta^r \Delta$ is also regular. Restricting its regular function to the subcomplex $K'$ we get $K'$ to be regular, as codimension one simplexes of $K'$ are also codimension one simplexes of $\beta^r \Delta$. As $|K|=|K'|$ so applying Lemma \ref{derivedtorefined} a second time, we get $s \in \N$ such that $\beta^s K$ is a subdivision of $K'$. Finally as $\beta^s K$ is the subdivision of a regular subdivision $K'$ of $K$ so by Claim 3 in proof of Theorem 1 of \cite{AdiIzm}, $\beta^s K$ is a regular triangulation.
\end{proof}

\begin{proof}[Proof of \ref{starconvexthm}]
The techniques in this proof are essentially those of Morelli and Wlodarczyk as detailed in Section 2 of \cite{IzmSch}. 

Choose $a \in \R^{n+1}$ outside $K_1$ such that the orthogonal projection map $pr: \R^{n+1} \to \R^n$ takes  the support of $C(K_1)=a\star K_1 \subset \R^{n+1}$ onto $P$. Furthermore $pr$ takes $a$ to the interior of an $n$-simplex of $K_1$ and to the interior of an $n$-simplex of $K_2$. By Lemma \ref{regular}, there exists $s\in \N$ so that $\beta^s C(K_1)$ is a regular simplicial cobordism between $\beta^s K_1$ and $\beta^s C(\del K_1)$ and similarly $\beta^s C(K_2)$ is a regular simplicial cobordism between $\beta^s K_2$ and $\beta^s C(\del K_2)$. Furthermore, as $\del K_1 = \del K_2$, by choosing the new vertices identically on $C(\del K_1)= C(\del K_2)$, we can ensure that $\beta^s C(\del K_1) = \beta^s C(\del K_2)$. We also choose the new vertices of the derived subdivisions such that for any simplex $A$ in $C(K_1)$ or $C(K_2)$ of dimension less than $n+1$, $pr(A)$ is a simplex of the same dimension as $A$.

Denote $\beta^s C(K_1)$ by $K$. Let $h:|K| \to \R$ be a regular function for $K$. If a simplex $\sigma'$ has some point above a simplex $\sigma$ then $\frac{\del h}{\del x_{n+1}}$ on $\sigma'$ is greater than $\frac{\del h}{\del x_{n+1}}$ on $\sigma$. So inductively removing simplexes in non-increasing order of the vertical derivative of $h$ we ensure that the projection of the upper boundary onto $P$ is always one-to-one.  That is, we get a sequence of triangulations $K=\Sigma_0$, $\Sigma_1$, ... , $\Sigma_N= K_1$ such that $\Sigma_{i+1}= \Sigma_i \setminus \sigma_i$ and the orthogonal projection $pr: \del^+ \Sigma_i \to P$ from the upper boundary of $\Sigma_i$ onto $P$ is one-to-one for every $i$. Removing an $n+1$-simplex $\sigma_i$ from $K$ corresponds to a bistellar move on $\del^+ \Sigma_i$. As the projection map is linear so it also corresponds to a bistellar move taking $pr(\del^+ \Sigma_i)$ to $pr(\del^+ \Sigma_{i+1})$. Therefore  $pr(\del^+ \Sigma_0)= \beta^s C(\del K_1)$ is bistellar equivalent to $pr(\del^+ \Sigma_N)= \beta^s K_1$. Consequently, $\beta^s K_1$ is bistellar equivalent to $\beta^s K_2$ via $\beta^s C(\del K_1)=\beta^s C(\del K_2)$.
\end{proof}

\section{Geometric manifolds}\label{geomsec}
\begin{definition}
Let $K$ be a geometric triangulation of a Riemannian manifold $M$ and let $L$ be a subcomplex of $K$. We call $K$ locally geodesically-flat relative to $L$ if for each simplex $A$ of $K\setminus L$, $st(A, K)\setminus lk(A, K)$ is simplicially isomorphic to the interior of a star-convex flat polyhedron in $\R^n$ by a map which takes geodesics to straight lines.
\end{definition}

\begin{definition}
Let $L$ be a (possibly empty) subcomplex of $K$ containing $\del K$ and let $\alpha K$ be a subdivision of $K$ which agrees with $K$ on $L$. Let $\beta^\alpha_r K$ be the subdivision of $K$ such that, if $A$ is a simplex in $L$ or $dim (A) \leq r$ , then $\beta^\alpha_r A = \alpha A$. If $A$ is not in $L$ and $dim (A) > r$ then $\beta^\alpha_r A = a \star \beta^\alpha_r \del \alpha A$ for $a$ a point in interior of $A$, i.e. it is subdivided as the cone on the already defined subdivision of its boundary. Observe that $\beta^\alpha_n K$ is $\alpha K$ while $\beta^\alpha_0 K = \beta_L K$ is a derived subdivision of $K$ relative to $L$.
\end{definition}

\begin{lemma}\label{locallyflatlemma}
Let $K$ be a locally geodesically-flat simplicial complex relative to a (possibly empty) subcomplex $L$ which contains $\del K$. Let $\alpha K$ be a geometric subdivision of $K$ which agrees with $K$ on $L$. Then there exists $s\in \N$ for which $\beta^s \alpha K$ is related to $\beta^s K$ by bistellar moves which keep $\beta^s L$ fixed.
\end{lemma}
\begin{proof}
For $A$ a positive dimensional $r$-simplex in $K\setminus L$, $st(A, \beta^\alpha_r K)$ is a strictly star-convex subset of $st(A, K)$. As $K$ is locally geodesically-flat relative to $L$, there exists a geodesic embedding taking $st(A, \beta^\alpha_r K)$ to a strictly star-convex flat polyhedron of $\R^n$. By Theorem \ref{starconvexthm}, $\beta^s st(A, \beta^\alpha_r K)$ is bistellar equivalent to $\beta^s C(\del st(A, \beta^\alpha_r K))$. As $A$ is not in $L$ so no interior simplex of $st(A, \beta^\alpha_r K)$ is in $L$ and consequently these bistellar moves keep $\beta^s L$ fixed. Taking all simplexes $A$ in $K \setminus L$ of dimension $r=n$, we get a sequence of bistellar moves taking $\beta^s \beta^\alpha_r K$ to $\beta^s \beta^\alpha_{r-1} K$. Ranging $r$ from $n$ down to $1$, we inductively obtain a sequence of bistellar moves taking $\beta^s \alpha K =\beta^s \beta^\alpha_n K$ to $\beta^s \beta_L K=\beta^s \beta^\alpha_{0} K$, which keeps $\beta^s L$ fixed.  

And finally, arguing as above with the trivial subdivision $\alpha K = K$, we get $\beta^s \beta_L K$ from $\beta^s K$ by bistellar moves which keep $\beta^s L$ fixed. 
\end{proof}

The following simple observation allows us to treat the star of a simplex in a geometric triangulation as the linear triangulation of a star-convex polytope in $\R^n$ and bistellar moves in the manifold as bistellar moves of the polytope.

\begin{lemma}\label{convex}
Let $K$ be a geometric simplicial triangulation of a spherical, hyperbolic of Euclidean $n$-manifold $M$ and let $L$ be a subcomplex of $K$ containing $\del K$. When $M$ is spherical we require the star of each positive dimensional simplex of $K\setminus L$ to have diameter less than $\pi$. When $M$ is cusped we include the ideal vertices in $L$. Then $K$ is locally geodesically-flat relative to $L$. 
\end{lemma}
\begin{proof}
Let $K$ be a geometric triangulation of $M$ and let $B$ be the interior of the star of a simplex in $K \setminus L$. As $K$ is simplicial, $B$ is an open $n$-ball. 

When $M$ is hyperbolic, let $\phi: B \to \H^n$ be the lift of $B$ to the hyperbolic space in the Klein model. As geodesics in the Klein model are Euclidean straight lines (as sets) so $\phi$ is the required embedding. 

When $M$ is spherical, let $D$ be the southern hemisphere of $\S^n \subset \R^{n+1}$, let $T$ be the hyperplane $x_{n+1}=-1$ and let $p: D \to T$ be the radial projection map (gnomonic projection) which takes spherical geodesics to Euclidean straight lines. As $B$ is small enough, lift $B$ to $D$ and compose with the projection $p$ to obtain the required embedding $\phi$ from $B$ to $T \simeq \E^n$. 

When $B$ is Euclidean let $\phi$ be the lift of $B$ to $\R^n$, which is an isometry.
\end{proof}

It is known (Theorem 4(c) of \cite{AdiIzm}) that for simplicial complexes of dimension at least 5 the number of derived subdivisions required to make the link of a vertex combinatorially isomorphic to a convex polyhedron is not (Turing machine) computable. So in particular, the stars of simplexes of a geometric triangulation may not even be combinatorially isomorphic to convex polyhedra, which is why we need to work with star-convex polyhedra instead.

Given a Riemannian manifold $M$, a \emph{geometric polytopal complex} $C$ of $M$ is a finite collection of geometric convex polytopes in $M$ whose union is all of $M$ and such that for every $P \in C$, $C$ contains all faces of $P$ and the intersection of two polytopes is a face of each of them.

\begin{proof}[Proof of \ref{mainthm1} and \ref{mainthm2}]
By Lemma \ref{convex}, $K_1$ and $K_2$ are locally geodesically flat simplicial complexes. Let $C$ be the geometric polytopal complex obtained by intersecting the simplexes of $K_1$ and $K_2$. Then $K=\beta_L C$, the derived subdivision of $C$ relative to $L$ is a common geometric subdivision of $K_1$ and $K_2$. When $M$ is a cusped manifold we assume that we are given such a common geometric subdivision $K$ as $C$ might have infinitely many polytopes. By Lemma \ref{locallyflatlemma} then, there exists $s\in \N$ so that $\beta^s K_1$ and $\beta^s K_2$ are bistellar equivalent via $\beta^s K$ by bistellar moves which leave $\beta^s L$ fixed. In the cusped situation we can take $L$ as the set of ideal vertices.
\end{proof}

\begin{acknowledgements}
The first author was supported by the MATRICS grant of Science and Engineering Research Board, GoI and the second author was supported by National Board of Higher Mathematics, GoI.
\end{acknowledgements}

\bibliographystyle{amsplain}

\end{document}